\documentclass[a4paper]{article}
\usepackage{fullpage,url,amsmath,amsfonts,amssymb, tedmath}
\usepackage{amsmath,amsfonts,amssymb, tedmath}

\newcommand{\lam}{\lambda}
\newcommand{\he}{\theta}

\newcommand{\tr}{\mbox{tr}}
\newcommand{\st}{\mbox{st}\,}

\title{Unique builders for  classes of matrices \footnote{MSC 2010 Classification: 15A23, 81P45, 94A15}
}
\author{
 Ted
 Hurley\footnote{National Universiy of Ireland Galway, email:
 Ted.Hurley@NuiGalway.ie }}
\date{}

\begin{document}

\maketitle




\begin{abstract}  
Basic matrices are defined which provide unique  building blocks for the class of normal matrices which include the classes of unitary and Hermitian matrices. 
Unique  builders for quantum logic gates are hence derived as a quantum logic gates is  represented by, or is said to be, a  unitary matrix.  
An efficient algorithm for expressing an idempotent as a unique sum of rank $1$  idempotents with increasing initial zeros is derived. This is used to derive a unique form for mixed matrices. 
A number of (further) applications are given: for example  (i) $U$ is a  symmetric unitary matrix if and only if it has the form $I-2E$ for a symmetric idempotent $E$, (ii) a formula for the pseudo inverse in terms of basic matrices is derived.  Examples  for various uses are readily available. 

\end{abstract}

\section{Introduction}

{\em Basic}  matrices are defined and it is
shown that any normal  matrix is the product of  basic commuting 
matrices  and that the  product is {\em unique} apart from the  order. 
A  matrix $B$ is normal when  $BB^*=B^*B$ where $B^*$ denote the complex conjugate transposed of $B$. The class of normal matrices include the classes of unitary matrices ($UU^*=I$), and Hermitian, also called self-adjoint,  matrices ($H^*=H$). These occur in many applications: for example quantum logic gates are represented by unitary matrices and their properties and applications depend ultimately on the structure of unitary matrices.

Each basic matrix itself is a product of {\em minimal basic} matrices with the same  eigenvalue.
A  basic matrix is expressed   in terms  of  a symmetric 
 idempotent matrix;  
 an idempotent $E$ occurring  in  the expression $B$  as a product  of basic matrices has the property that $B$ acts on $E$  
by $BE=\al E$ and the 
eigenvalue $\al $ occurs with multiplicity equal to  the rank of $E$.

An efficient algorithm is given for expressing a rank $r$ idempotent matrix as a sum of $r$ rank $1$ (pure) orthogonal idempotents with increasing initial zeros and such an  expression is unique. From this a unique expression for a mixed matrix as a sum of rank $1$ idempotents of this type is  derived. 


 A quantum logic gate is 
 represented
by a unitary matrix  and these  gates are the basis for
quantum information theory, see  for example \cite{nielsen}. Indeed a quantum logic gate is often stated to be a unitary matrix itself. 
A quantum logic gate is thus a unique, apart from order,  product
of {\em basic} logic gates and these basic logic
gates are building blocks for quantum 
logic gates in general.   

Examples are readily available and applications are given throughout:
It is shown that  a unitary matrix $U$ is symmetric if and only if it has the form $U=I-2E$ for a symmetric idempotent $E$. An easy  formula for the roots and powers of the matrices  follows directly from its expression as a product of basic matrices; powers and roots  are given explicitly in terms of basic matrices. A formula for the pseudo inverse is immediate. 


Expressions  for  well-known quantum gates (such as Pauli, Hadamard gates) as products of unique basic matrices   
are explicitly derived in section \ref{open}.  

Comparisons may  be made with the famous 1D factorization theorem of Belevich and Vaidyanathan which derives building blocks for 1D paraunitary matrices,  \cite{1D} pp. 302-322.     
The  basic matrices derived here for
building normal matrices are influenced by methods  in 
\cite{tedbarry} for constructing generators/builders for multidimensional {\em paraunitary} matrices\footnote{A paraunitary matrix is a matrix
$U(\bf z)$ such that $U({\bf z})U^*({\bf z^{-1})}=I$ where ${\bf z}=(z_1,z_2,\ldots,z_k)$ and ${\bf z^{-1}} =
  (z_1^{-1},z_2^{-1}, \ldots, z_k^{-1})$.}
, and in particular for constructing paraunitary
 non-separable (entangled) matrices. 




A summary of the main results  is given   in Section \ref{summary}.  Further notation from  Section \ref{back} may be consulted  as required. 

\subsection{Summary}\label{summary} 

Let $B$ be a normal matrix. Then $B$ is a product 
$B= \prod_{j=1}^k(I- E_j+ \al_j E_j)$ the $\al_j$ are {\em distinct} and $S=\{E_1, E_2, \ldots, E_k\}$ is an orthogonal symmetric set of idempotents.   Each $(I-E_j+\al E_j)$ is termed a {\em basic matrix} and the product  is unique apart  from the order of these commuting  basic matrices.  

$B$ acts on the idempotents by  $BE_j=\al_j E_j$ and so $\al_j $ is an eigenvalue of $B$ occurring to multiplicity equal to $\rank E_j$. When $S$ is not a complete set then $E=
(I-\sum_{j=1}^kE_k)$ completes the set $S$ and $1$ occurs as an eigenvalue of $B$ with multiplicity equal to $(n- \rank(\sum_{j=1}^kE_j)) = n - (\sum_{j=1}^k\rank E_j)$ where $n$ is the size of the matrices.

A basic matrix $(I- E +\al E)$ is the product of minimal basic matrices  $\prod_{j=1}^k(I-E_j+\al E_j)$ where the $E_j$ are mutually orthogonal, have rank $1$ and each minimal basic has the same eigenvalue $\al$ as the original. 

The idempotent matrix $E$ is of  the form $E=\sum_{j=1}^tu_ju_j^*$ for mutually orthogonal unit column vectors $u_j$. Given an idempotent matrix $E$ of rank $r$ an efficient algorithm is given for expressing $E$ as a sum of such rank $1$ idempotents with increasing initial zeros so that the expression obtained is unique. 
Mixed density matrices have a particularly useful  form when viewed as products of basic matrices. A mixed density matrix is defined as a convex sum of pure density matrices though not in a unique way. By writing the mixed matrix  as a product of basic matrices, uniqueness is obtained and this gives a unique perspective on mixed matrices. Further each basic matrix is a unique product of ordered rank $1$ basics matrices giving a unique expression for a mixed density matrix.  See Section \ref{mixed}.

Unitary matrices are normal matrices with eigenvalues of the form $e^{i\theta}$. Quantum logic gates are represented by unitary matrices; indeed a quantum logic gate is often {\em defined}  to be a   unitary matrix. The expression of a unitary matrix as a unique
product, except for order,  of basic matrices then  expresses a quantum logic gate as a unique product of basic quantum logic gates. Each basic quantum logic gate is a product of minimal quantum logic gates with the same eigenvalue.  



A number of easy consequences are derived. A formula  
for the pseudo inverse of the matrix follows directly from the expression of the  matrix as a product of basic matrices.  
  It is shown that  $U$ is a symmetric unitary matrix if and only if $U=I-2E$ for a idempotent $E$; thus building symmetric unitary matrices from  idempotent  matrices is straight forward.
Many of the mostly used quantum logic gates are symmetric.
In Section \ref{root} a useful direct formula for writing down all the roots of a normal matrix  is derived giving in particular  a useful direct formula for  roots of  a unitary matrix/logic gate. 
Section \ref{open} finds expressions for  common logic gates as products of  basic logic gates. 
Section \ref{build1} discusses techniques for building normal matrices including unitary matrices from basic matrices.

  There are easily constructed interesting examples; those displayed are of small size but the constructions can be efficiently applied to large size matrices. 

 \subsection{Additional  notation}\label{back} Necessary background on 
algebra may be   found in  many algebra and  linear algebra books.
Background on quantum information theory may be found in \cite{nielsen}. 

$R$ denotes a ring with identity $I_R$; the suffix $_R$ may  be
omitted when a particular ring $R$ is understood. A mapping $^*: R\to R$ in which $r\mapsto r^*, (r\in R)$ 
is said to be an {\em involution} on $R$ if and only if (i) $r^{**} = (r^*)^* = 
r, \, \forall r\in R$, (ii) $(a+b)^* =
a^*+b^*, \, \forall a,b \in R$, and 
(iii) $(ab)^* = b^*a^* , \, \forall a,b \in R$. 
Let $R$ be a ring with involution $^*$. An element  $a\in R$ is said to be {\em symmetric}, with respect to $^*$, if $a^*=a$.  An idempotent in $R$ is an element $E$ such that $E^2=E$. $E,F$ are said to be orthogonal in $R$ if $EF=FE=0_R$. The set  $\{E_1, E_2, \ldots, E_k\}$ is said to be a {\em complete set of orthogonal idempotents} in $R$ if each element is an idempotent, the $E_i$ are mutually orthogonal and $E_1+E_2 + \ldots + E_k = I_R$. The set is further said to be symmetric if each $E_i$ is symmetric (with respect to $^*$).

Here we  work in the field of complex numbers $\C$ although many
of the results work over other systems but are not included.  
For $a\in \C$,
$a^*$ denotes the complex conjugate of $a$ and  then $A^*$ denotes the
complex conjugate transposed of $A$ for $A\in \C_{n\ti m}$ or $A\in
\C^n$. Now $I_n$ denotes the identity $n\ti n$ matrix; the suffix $_n$ will be omitted when the size is clear. 
  An $n\ti n$  matrix $B$ is said to be a normal matrix in $\C$ if and only if $BB^*=B^*B$. 
An $n\ti n$  matrix is unitary if and only if $UU^*=I_n$. An $n\ti n$ matrix is Hermitian (or self-adjoint) if and only if $H^*=H$. Thus unitary and Hermitian matrices are normal matrices. Unitary matrices by their definition are invertible but an Hermitian matrix may not be invertible. 


Column vectors $u,v$ are orthogonal if $u^*v=0$;  $n\ti n$ matrices
are orthogonal if $AB^* = 0_{n\ti n} = BA^*$; in all cases here orthogonality will refer to symmetric matrices. For a column $n\ti 1$ vector $v\neq 0$,  $vv^*$ is a
symmetric $n\ti n$ matrix  and the matrix is necessarily of rank $1$. 
When
$v$ is a unit vector ($v^*v=1$) then $vv^*$ is an idempotent matrix as
then $vv^*(vv^*)^* = vv^*vv^* = vv^*$. Suppose $v,w$ are orthogonal
vectors. Then the matrices $vv^*$ and $ww^*$ are orthogonal matrices as $vv^*ww^*
= 0_{n\ti n}$ since $v^*w=0 $.





\section{Basic Matrices}\label{basic}
A matrix $A$ is normal if and only if there exists a unitary matrix $P$ such that $P^*AP=D$ where $D$ is a diagonal matrix. Normal matrices include   Hermitian matrices ($H^*=H$) and unitary matrices ($UU^*=I$); Hermitian matrices include real symmetric matrices. 
Eigenvalues of a unitary matrix have the form $e^{i\theta}$ and an Hermitian matrix has real eigenvalues. 

\begin{proposition}\label{gen} Let $A$ be an $n\ti n$ matrix. There exists a unitary matrix $P$ such that  $P^*AP=D$, where $D$ is diagonal, if and only if $A=\al_1v_1v_1^*+ \al_2v_2v_2^* + \ldots + \al_nv_nv_n^*$ for an orthogonal set of unit vectors $\{v_i | i=1,2, \ldots, n\}$.
\end{proposition}

This can  be restated as follows:

\begin{proposition}\label{gen1} Let $A$ be an $n\ti n$ matrix. Then $A$ is a normal matrix  if and only if $A=\al_1v_1v_1^*+ \al_2v_2v_2^* + \ldots + \al_nv_nv_n^*$ for an orthogonal set of unit vectors $\{v_i | i=1,2, \ldots, n\}$.
\end{proposition}
 
The $v_i$ consist of the columns of $P$ and the $\al_i$ are the diagonal entries of $D$ in Proposition \ref{gen}  and are eigenvalues of $A$. Now $\{v_1,v_2, \ldots, v_n\}$ is an orthonormal basis for $\C_n$. Define $E_i=v_iv_i^*$. Then $\{E_1,E_2, \ldots, E_n\}$ is an orthogonal symmetric set of idempotents. The set is also a complete set of idempotents for if not then $E=I-\sum_{i=1}^nE_i \neq 0$ and $E$ is orthogonal to each of the $E_i$; this would give a linearly independent set of $n+1$ vectors in the $n$-dimensional space $\C_n$.

Now $\tr A$ denotes the trace of the matrix $A$ which is the sum of its  diagonal elements. One  nice property of an idempotent matrix is that its rank is the same as its trace, see for example \cite{idemrank}. 
\begin{lemma}\label{rank} Let  $E,F$ be  orthogonal symmetric idempotent
 $n\ti n$ matrices. Then
\begin{enumerate} \item $E+F$ is a symmetric  idempotent matrix. 
  \item  If $E+F \neq I$, then $I-E-F$ is a symmetric  idempotent matrix orthogonal to both $E$ and $F$. 
  \item\label{3}  $\rank (E+F) = \rank E + \rank F$.
  \item\label{4}  $\rank (I-E-F) = n- \rank E -\rank F$. 
    \end{enumerate}
\end{lemma}

\begin{proof} The proofs of the first two items are straight forward.
  Proof of \ref{3}: It is known that $\rank E = \tr E $ for an idempotent matrix $E$, see for example \cite{idemrank}. 
  Now $E+F$ is an idempotent and so $\rank (E+F) = \tr (E+F) = \tr E + \tr F = \rank E + \rank F$.
The proof of \ref{4} is similar. 
\end{proof}

Lemma \ref{rank} may be generalised as follows: 
\begin{lemma}\label{trrank} Let  $\{E_1,E_2, \ldots, E_s\}$ be a
  set of orthogonal symmetric  idempotent matrices. Then
  \begin{itemize} \item  $(E_1+E_2 +\ldots + E_s)$ is an idempotent symmetric
    matrix.
  \item  If $(E_1+E_2 +\ldots + E_s) \neq I$ then $I- (E_1+E_2 +\ldots + E_s) $ is a symmetric  idempotent orthogonal to each $E_j$ for $ j= 1,2,\ldots, s$.
  \item $\rank (E_1+E_2+ \ldots + E_s)
 = \tr (E_1+E_2+ \ldots + E_s)   = \tr E_1+ \tr
 E_2+ \ldots + \tr E_s = \rank E_1+ \rank E_2 + \ldots +\rank
 E_s$.
\item  $\rank (I- (E_1+E_2 +\ldots + E_s))  = n- \rank E_1 -\rank E_2 -\ldots - \rank E_s$.
  \end{itemize}
\end{lemma}

\begin{lemma}\label{mut1} Let $\{E_1,E_2, \ldots, E_k\}$ be mutually
 orthogonal symmetric idempotent matrices. Then
 $(I-E_1+\al_1E_1)(I-E_2+\al_2E_2) \ldots
 (I-E_k+\al_kE_k) =  \al_1E_1+\al_2E_2  + \ldots + \al_kE_k + (I- E_1 -E_2 - \ldots - E_k)$.
\end{lemma}
The following follows directly from Lemma \ref{mut1}. 
\begin{proposition}\label{help} Let $A$ be an $n\ti n$ matrix with  
  $A=\al_1E_1+ \al_2E_2 + \ldots + \al_nE_n$ where $\{E_i | i=1,2,\ldots, n\}$ is a complete orthogonal set of symmetric idempotents. Then  $A= (I-E_1+\al_1E_1)(I-E_2+\al_2E_2) \ldots
  (I-E_n+\al_nE_n) $.
\end{proposition}
 
A {\em basic} matrix is one of the form $(I-E+\al E)$ where $E$ is an idempotent. This has eigenvalue $\al$ occurring to multiplicity equal to $\rank E$ and has eigenvalue $1$ occurring to multiplicity equal to $\rank (I- E) = n-\rank E$, where $n$ is the size of $E$. When $\al \neq 0$,  it  has inverse $(I-E+\al^{-1}E)$. 
\begin{lemma}\label{join} Let  $E,F$  be orthogonal symmetric
  idempotent matrices. Then 
$(I-E+\al E)(I-F+\al F) = I- (E+F)+ \al (E+F)$
\end{lemma}

Lemma \ref{join} enables basic matrices with orthogonal idempotents and the same eigenvalue to be collected together. 

The following is a consequence of Lemma \ref{join} and Proposition \ref{help}. 

\begin{proposition}\label{not} Let $A$ be an $n\ti n$ normal matrix. Then  $A= (I-E_1+\al_1E_1)(I-E_2+\al_2E_2) \ldots   (I-E_k+\al_kE_k) $, the $\al_i$ are distinct and $\{E_1, E_2, \ldots, E_k\}$ is a  complete orthogonal set of idempotents. Moreover the eigenvalues of $A$ are $\{\al_i | i= 1,2,\ldots, k\}$ and $\al_i$ occurs with multiplicity equal to $\rank E_i$.
\end{proposition}

It is shown below, Proposition \ref{unique}, that the expression for $A$ in Proposition \ref{not} is {\em unique} apart from the order of the commuting basic matrices. Note that it is required for this uniqueness that the $\al_i$ are distinct; the equal $\al_j$ have been  gathered into one basic matrix by Lemma \ref{join}.

Each idempotent $E$ is a sum of rank $1$ idempotents but not in a unique manner. Algorithm \ref{pro} and Theorem  \ref{nut} below give a method of expressing an idempotent $E$ as the sum of rank $1$ (pure) idempotents which have increasing initial zeros thus giving  a {\em unique} expression for $E$ as a sum of such  pure idempotents.   

One or more of  the $\al_j$ in Proposition \ref{not} may be $0$. 
Some of the $\al_j$ may be also $1$ and then $(I-E_j+E_j)=I$ giving that $A= (I-E_1+\al_1E_1)(I-E_2+\al_2E_2) \ldots   (I-E_k+\al_kE_k) $ where the $\al_i$ are distinct and $\neq 1$.  Here then the eigenvalue $1$ is in `disguise' and if so, it  occurs with multiplicity equal to $(n- \sum_{j=1}^k\rank E_j)$.

Now $A$ {\em acts} on its idempotents by $AE_j=\al_jE_j$, and  $AE=E$ when $E=(I - \sum_{j=1}^{k}E_j) \neq 0$.

In particular Proposition \ref{not} can be applied  to unitary matrices. The eigenvalues of a unitary matrix are of the form $e^{i\theta}$.
\begin{proposition}\label{notunit} Let $U$ be an $n\ti n$ unitary matrix.  Then  $U= (I-E_1+e^{i\theta_1}E_1)(I-E_2+e^{i\theta_2}E_2) \ldots   (I-E_k+e^{i\theta_k}E_k) $, the $\al_i$ are distinct and $\{E_1, E_2, \ldots, E_k\}$ is an orthogonal set of idempotents. Moreover the eigenvalues of $U$ are $\{e^{i\theta_i}, i= 1,2,\ldots, k\}$ and $e^{i\theta_i}$ occurs with multiplicity $\rank E_i$. Now  eigenvalue $1$ occurs with multiplicity $n-(\sum_{j=1}^k\rank E_j)$ which may be $0$. Moreover a matrix of the form $(I-E_1+e^{i\theta_1}E_1)(I-E_2+e^{i\theta_2}E_2) \ldots   (I-E_k+e^{i\theta_k}E_k) $ is a unitary matrix for any orthogonal set of symmetric idempotents $\{E_1, E_2, \ldots, E_k\}$. 
  \end{proposition}

A basic unitary matrix is one of the form $(I-E+e^{i\theta}E)$. 

 In Proposition \ref{not}, $A$ may be singular in which case the eigenvalue $0$ appears to a certain multiplicity. 
In this case
$A= (I-E_1+\al_1E_1)(I-E_2+\al_2E_2) \ldots   (I-E_k+\al_kE_k)(I-E) $, where the $\al_i$ are distinct and  $\neq 0$ and $\{E_1, E_2, \ldots, E_k, E\}$ is an orthogonal set of idempotents. (The eigenvalue $1$ is {\em hidden} when  $\{E_1, E_2, \ldots, E_k, E\}$ is not complete.)
It is easy to check that $A^+=\prod_{j=1}^k(I-E_j+\al_j^{-1}E_j)(I-E)$ is the {\em pseudo inverse} of $A$. 
\begin{proposition} Let $A= (I-E_1+\al_1E_1)(I-E_2+\al_2E_2) \ldots   (I-E_k+\al_kE_k)(I-E) $, where the $\al_i$ are distinct and  $\neq 0$ and $\{E_1, E_2, \ldots, E_k, E\}$ is an orthogonal set of idempotents. Then $A^+=\prod_{j=1}^k(I-E_j+\al_j^{-1}E_j)(I-E)$ is the {\em pseudo inverse} of $A$.
  \end{proposition}

\begin{proposition}\label{huy} Let $A$ be a normal matrix. Then $A^2=I$ if and only $A=(I- 2E)$ for a symmetric idempotent $E$. 
\end{proposition}
\begin{proof} Suppose $A^2=1$. Then the eigenvalues of $A$ are $\{-1,1\}$. Thus by Proposition \ref{not},  $A= (I-E-E)(I- F+F) = (I-2E)$ for symmetric orthogonal idempotents $E,F$. If $A= (I-2E)$ then clearly $A^2=I$ and $A$ can be diagonalised by a unitary matrix. 
\end{proof}

The following is a corollary:
\begin{proposition}\label{cor} $U$ is a symmetric ($U^*=U$) unitary matrix if and only if $U=I-2E$ for an idempotent $E$.
\end{proposition}

Notice in  Proposition \ref{huy}  and in Proposition \ref{cor}, the eigenvalue $\{-1\}$ occurs with multiplicity equal to $\rank E$ and the eigenvalue  $\{1\}$ occurs with multiplicity equal to $\rank F= n-\rank E$.
\begin{corollary} Let $HH = nI $ and $H=H^*$ for an $n\ti n$ matrix $H$. Then $H=\sqrt{n}(I-2E)$ for a symmetric idempotent $E$. In particular this applies to a symmetric Hadamard matrix.
\end{corollary}

\subsection{Examples}\label{example}
\begin{enumerate}
\item\label{oneone} $U= \begin{ssmatrix} 0 & 1 \\ 1 & 0 \end{ssmatrix}$. This has eigenvalues $\{1,-1\}$ and corresponding eigenvectors $(1,1)\T, (1,-1)\T$. Then 
$U= E_1 -E_2 = E_1 +e^{i\pi}E_2 = (I-E_1+E_1)(I- E_2 + e^{i\pi}E_2) = I - E_2 +e^{i\pi}E_2$ where $E_1 = \frac{1}{2}\begin{ssmatrix} 1 & 1 \\ 1 & 1 \end{ssmatrix},
E_2 = \frac{1}{2}\begin{ssmatrix} 1 & -1 \\ -1 & 1 \end{ssmatrix}$. 
Using this formula,  roots of $U$ may be obtained directly -- see Section \ref{root}. 

\item\label{two} A  commonly referenced matrix is the following real orthogonal/unitary matrix
$U=\begin{ssmatrix} \cos \theta & \sin \theta \\ -\sin \theta & \cos
    \theta \end{ssmatrix}$.
This has eigenvalues $\{e^{i\theta}, e^{-i\theta}\}$. 
Then 
$U=e^{i\theta}E_1+ e^{-i\theta}E_2 =
\frac{1}{2}e^{i\theta}\begin{ssmatrix}1 & -i \\ i & 1 
		      \end{ssmatrix} + \frac{1}{2}e^{-i\theta}
\begin{ssmatrix}1 & i \\ -i & 1
\end{ssmatrix} $, 
which may be checked independently. The $E_i$ are symmetric
idempotents. 
It is seen that 
$U=(I-E_1+e^{i\theta}E_1)(I-E_2+e^{-i\theta}E_2)$ 
which gives $U$ as a product of {\em basic} unitary matrices, which are defined in Section \ref{build} below. 
\end{enumerate}

\subsection{Uniqueness}\label{build}

 \begin{definition} A {\em basic matrix} is one  of the form $B= (I-E+ \al E)$ for a symmetric idempotent matrix $E$.
\end{definition}
The idempotent $E$ of $B$ is said to be the idempotent {\em involved in  $B$} and $\al $ is the eigenvalue {\em involved in $B$}.

\begin{definition} Say a basic matrix $I-E+\al E$ is a {\em simple} basic unitary matrix if the idempotent $E$ has $\rank 1$.
  \end{definition}
\begin{proposition}\label{one} A symmetric  idempotent $E$ has $\rank 1$ if and only if $E=uu^*$  
  for a unit column vector $u$.
\end{proposition}

\begin{proof} If $E$ has the form $uu^*$ for a unit column vector $u$,  then clearly $E$ is a symmetric idempotent of $\rank 1$.
  On the other hand if $E$ is a symmetric idempotent of $\rank 1$ then
 $E$ is of the form $uu^*$ for a unit column vector $u$. This is shown
 in for example 
 \cite{tedbarry} Proposition 4.5.
  \end{proof}

  A minimal basic is of the form
 $(I-E+\al E)$ for $E=uu^*$ where $u$ is a unit column vector and
 has rank $1$ and $\al\neq 0$. 

\paragraph{Note} Let $I-E+\al E$ and $I-E+\be E$ be basic
 matrices with the same idempotent. Then 
$(I-E+\al E)(I-E+\be E)= I-E + \al \be E$. A product of a minimal basic with another minimal basic
with the same idempotent is another minimal basic.  
The minimality is expressed in terms of the {\em rank} of the idempotent involved.

\begin{proposition}\label{cols} A symmetric idempotent $E$ has $\rank k$ if and only if $E=u_1u_1^* + u_2u_2^* + \ldots + u_ku_k^*$ for unit column mutually orthogonal vectors $u_i$.
\end{proposition}


\begin{proof} 
  When $k=1$ the result follows from
 Proposition \ref{one}. If the first column of $E$ is zero
 then the first row of $E$ is zero and the result follows by induction. 
Suppose then the first column $w$ of $E$ is non-zero and define $u=w/|w|$ which is a unit column vector.  Then $F= uu^*$ is an idempotent of rank $1$. Let $E_1=
  E- F$. Then $E_1= E-F$ is an idempotent symmetric matrix orthogonal to
 $E$ which has first column, and hence first row, consisting of zeros. Let
 $A$ be the $(n-1)\ti (n-1)$ matrix with first column and first row of
 $E_1$ omitted.  Since $F,E_1$ are orthogonal it follows that $\rank F +
 \rank E_1 = \rank (F+E_1) = \rank E$ and so $E-F$ has $\rank
 (k-1)$. Thus $A$ is rank $(k-1)$ symmetric idempotent matrix. The result then follows by induction.
  
\end{proof} 

Thus a basic unitary matrix is of the form $(I-E+\al E)$ where
$E=u_1u_1^* + u_1u_2^* + \ldots + u_ku_k^*$ for unit column mutually
orthogonal vectors $u_i$. When $k=1$ the basic unitary matrix is what is
termed a simple basic unitary matrix and its idempotent is of the form
$uu^*$. 

\begin{proposition}\label{unit} Suppose the space $W$ is generated by the unit
 orthogonal columns vectors $\{w_1,w_2,\ldots, w_k\}$ and by the unit orthogonal
 column vectors $\{v_1,v_2,\ldots, v_k\}$. Then 

$w_1w_1^*+w_2w_2^* + \ldots + w_kw_k^* = v_1v_1^* + \ldots + v_kv_k^*$.

\end{proposition}
\begin{proof} 
Now \begin{eqnarray*} w_1 = \al_{11}v_1 + \al_{12}v_2 + \ldots + \al_{1k}
 v_k \\ w_2 = \al_{21}v_1+ \al_{22}v_2 + \ldots + \al_{2k}v_k \\  \vdots
  \\ w_k=\al_{k1}v_1+ \al_{k2}v_2 + \ldots + \al_{kk}v_k \end{eqnarray*}

for some $\al_{ij} \in \C$.

In other words  \begin{eqnarray}\label{mark} (w_1, w_2, \ldots, w_k) = (v_1,v_2, \ldots,
 v_k)\begin{pmatrix}\al_{11}& \al_{21} & \ldots & \al_{k1} \\ \al_{12} &
      \al_{22} & \ldots & \al_{k2} \\ \vdots & \vdots & \vdots & \vdots
      \\ \al_{1k} & \al_{2k} & \ldots & \al_{kk}\end{pmatrix} \end{eqnarray}

Denote the matrix on the right (involving $\al_{ij}$) by $U$.

Then \begin{eqnarray}\label{to}(w_1,w_2,\ldots, w_k)\begin{pmatrix}w_1^*
						     \\ w_2^* \\ \vdots \\
						    w_k^*\end{pmatrix} =
						    (v_1,v_2,\ldots,
						    v_k)UU^*\begin{pmatrix}v_1^*
						     \\ v_2^*\\ \vdots
						      \\ v_k^*\end{pmatrix}\end{eqnarray}

Then  by (\ref{mark}),  \begin{eqnarray*} I_k&=& \begin{pmatrix}w_1^* \\ w_2^* \\ \vdots \\ w_k^* \end{pmatrix} 
 (w_1,w_2,\ldots,w_k) \\ &=& U^* \begin{pmatrix}v_1^* \\ v_2^* \\ \vdots \\ v_k^* \end{pmatrix}(v_1,v_2,\ldots, v_k) U  
 \\ &=&
 U^*U \end{eqnarray*}

Thus $U^*U = I_k$ and hence $UU^*=I_k$. Therefore by (\ref{to})

$(w_1,w_2,\ldots, w_k)\begin{pmatrix}w_1^*\\ w_2^* \\ \vdots \\
		      w_k^*\end{pmatrix} = (v_1,v_2, \ldots,
 v_k)\begin{pmatrix}v_1^* \\ v_2^* \\ \vdots \\ v_k^*\end{pmatrix} $
  
and hence $w_1w_1^*+w_2w_2^* + \ldots + w_kw_k^* = v_1v_1^* + \ldots + v_kv_k^*$.

\end{proof}

Suppose now $B=(1-E_1+\al_1E)(1-E_2+\al_2E_2)\ldots
(1-E_k+\al_kE_k)$
is a product of basic $n\ti n$ matrices. Implicit in this is that if $r=
\rank
E_1+ \rank E_2 + \ldots + \rank E_k = n - \rank(E_1+E_2 + \ldots + E_k) < n$ then $F=(I-E_1 -E_2 - \ldots
-E_k)$ is a non-zero idempotent with eigenvalue $1$ of  multiplicity $(n-r)$.




By Lemma \ref{join}  a product basic matrices using orthogonal idempotents with the same eigenvalues can be collected together
into  a  basic matrix.

 Suppose  $B=
(I-F_1+\al_1F_2)(I-F_2+\al_2F_2) \ldots (I-F_l+\al_lF_l)$ where
the $F_j$ are mutually orthogonal  idempotents and the $\al_j$ are all different. We now show
that except for the order such an expression is unique. Each $F_j$ in
the product has the form $w_1w_1^*+ w_2w_2^*+ \ldots + w_kw_k^*$ where
$k$ is the rank of $F_j$ and is also the multiplicity of the
corresponding eigenvalue $\al_j$.

\begin{proposition}\label{unique} Suppose $B=(I-E_1+\al_1E_1)(I-E_2+\al_2E_2)(\ldots
 )(I-E_k+\al_kE_k)=
(I-F_1+\be_1F_1)(I-F_2+\be_2F_2)(\ldots )(I-F_s+\be_sF_s)$
 where the $E_j$ are orthogonal idempotents and the
 $\al_j$ are distinct and where the $F_l$ are orthogonal idempotents 
  and the $\be_j$ are distinct. Then $s=k$ and by
 reordering $F_t=E_t, \al_t=\be_t$.
\end{proposition}
\begin{proof} Now $B$ has eigenvalues $\al_j$ occurring with
 multiplicity $\rank E_j$ and eigenmatrix $E_j$ and looking at it another way has eigenvalues
 $\be_k$ occurring with multiplicity $\rank F_k$. Thus $\al_1$
 must equal $\be_k$ for some $k$. We may assume this $k=1$ by
 reordering. Then $\rank E_1=\rank F_1$. Now $E_1 = v_1v_1^*+ v_2v_2^* + \ldots + v_tv_t^*$ for
 orthogonal unit $v_j$,  ($t=\rank E_1=\rank F_1$) and $F_1 = w_1w_1^* + w_2w_2^* +
 \ldots + w_tw_t^*$  for orthogonal unit vectors $w_j$. Now by Proposition
 \ref{unit},  $E_1 = F_1$. Similarly by reordering if necessary $E_j=F_j, \al_j=\be_j$ in general for $j=1,2, \ldots, k$ and of necessity $s=k$. \footnote{Induction is more complicated because we cannot assume  $(I-E_1+\al_1E)$ is invertible.} 
\end{proof} 

\section{Pure and mixed idempotents}

In order to obtain a  normal matrix as  a unique product of {\em basic} matrices it is necessary to collect the basics with the same eigenvalue, see Propositions \ref{not} and \ref{unique}. An idempotent is not a unique sum of idempotents of rank $1$. An idempotent of rank $1$ is often described as a {\em pure} idempotent by analogy with the mixed matrix in quantum theory, see Section \ref{mixed} below. 

It is now shown that  an idempotent matrix $E$ of rank $r$ may be written as the sum of special idempotents of rank $1$ in a {\em unique} way and an efficient algorithm is given.   

Let $E$ be an $n\ti n$ idempotent matrix of rank $r$.
Then $(I-E)$ is an idempotent matrix of rank $(n-r)$. Also $E$ has eigenvalues $\{1,0\}$.
Now $E.E= E $ and so the multiplicity of the eigenvalue $1$ is $\geq r$ since $E$ has rank $r$. Also $E.(I-E) = 0 = 0(I-E)$ and so the multiplicity of the eigenvalue $0$ is $\geq (n-r)$ as $(I-E)$ has rank $(n-r)$. Hence the multiplicity of the eigenvalue $1$ of $E$ is exactly $r$ and the multiplicity of the eigenvalue $0$ of $E$ is exactly $(n-r)$.

The following gives $E$ as a sum of rank $1$  idempotents from the column space of $E$. 

\begin{proposition}\label{idm} Let $\{e_1,e_2, \ldots , e_r\}$ be an orthonormal  basis for the column space of $E$. Define $E_i=e_ie_i^*$ for $i=1,2,\ldots, r$. Then
  $E= E_1+E_2+\ldots + E_r$.
\end{proposition}

\begin{proof} The space of eigenvectors of $E$ corresponding to the eigenvalue $\lam = 1$ is the solution space of $(E-1I)\underline{v}=0$ which is the solution space of $(I-E)\underline{v}=0$ and this has dimension $r$. Now $(I-E)E=0$ and so the columns of $E$ form a basis for this solution space. Let $\{v_1,v_2, \ldots,, v_r\}$ be an orthonormal basis for this solution space which is an orthonormal basis for the columns of $E$.

  The space of eigenvectors of $E$ corresponding to the eigenvalue $\lam = 0$ is the solution space of $(E-0I)\underline{v}=0$ which is the solution space of $E\underline{v}=0$ and this has dimension $(n-r)$. Now $E(I-E)=0$ and so the columns of $I-E$ form a basis for this solution space. Let $\{v_{r+1}, v_{r+2}, \ldots, v_n\}$ be an orthonormal  basis for this solution space which is an orthonormal
  basis for the column space of $(I-E)$.

  Now vectors corresponding to distinct eigenvalues of the symmetric matrix $E$ are othogonal and so $\{v_1,v_2, \ldots, v_n\}$ is an orthonormal  basis for $\C_n$. Let $P=(v_1,v_2,\ldots, v_n)$. Then $P$ is a unitary matrix and

  $P^*EP = \diag(1,1, \ldots, 1, 0,0, \ldots, 0)$ where the $1$ occurs $r$ times and  the $0$ occurs $(n-r)$ times on the diagonal.

  Then $E = v_1v_1^* + v_2v_2^* + \ldots + v_rv_r^*$ as required.
\end{proof}

This expression for $E$ is not unique as any orthonormal basis for the column space of $E$ may be used.

An efficient algorithm for finding $E$ as a sum of idempotents of rank $1$ is then formulated as follows:

\begin{Algorithm}\label{pro}

  \begin{enumerate}

  \item\label{tee}

    Let $v$ be the first  non-zero column of $E$ which occurs at column $r\geq 1$ so that columns $\{1, \ldots, r-1\}$  consist of zeros.
\item Define $u=\frac{v}{||v||}$ so $u$ is a unit column vector.
\item Define $E_1=uu^*$. Then $E_1$ is rank $1$ symmetric idempotent. 
\item Let $F = E-E_1$.  If $F=0$ then $E=E_1$ and $E$ has rank $1$ and we are finished. If $F\not = 0$  then $F$ is a nonzero idempotent with   columns $1,2, \ldots, r$ all  zero vectors. Now let $E$ be replaced by $F$ and go back to item \ref{tee}.

  \end{enumerate}

\end{Algorithm}

 To apply the algorithm it is not necessary to know the rank of the idempotent beforehand and this comes out  at the end. The algorithm is efficient and at each step just involves forming an idempotent from the first non-zero column of a matrix and subtracting this from the matrix.   



\paragraph{Example 1} $E= \frac{1}{3} \begin{pmatrix} 2 &-1&1\\ -1 & 2 & 1 \\ 1 &1 & 2 \end{pmatrix}$. Then $E$ is a rank $2$ idempotent. The first column of $E$ is non-zero so start there and let

$E_1= \frac{1}{6}\begin{pmatrix} 2\\-1\\ 1 \end{pmatrix}(2,-1,1) =\frac{1}{6}
\begin{pmatrix}4 &-2&2 \\ -2 &1&-1 \\ 2&-1&1 \end{pmatrix}$

Then $E-E_1 = \frac{1}{2} \begin{pmatrix} 0 & 0 &0 \\ 0 & 1 & 1 \\ 0 & 1 & 1 \end{pmatrix} = E_2$. Now $E_2$ is an idempotent of rank $1$ and $E=E_1+E_2$.

\paragraph{Example 2} Let $E = \frac{1}{4}\begin{pmatrix} 3 & i & 1 & -i \\ -i & 3 & i & 1 \\ 1 & -i & 3 & i \\ i & 1 & -i & 3 \end{pmatrix}$.

Apply  the algorithm and get in turn $E_1,E_2, E_3$ idempotents of rank $1$ and  
$E= E_1+E_2+E_3$
where

$E_1= \frac{1}{12}\begin{pmatrix} 9 & 3i &3 & -3i \\ -3i &1& -i& -1\\
  3 & i& 1 &-i \\ 3i& -1 & i & 1 \end{pmatrix} $,
$E_2 = \frac{1}{6}\begin{pmatrix} 0 & 0 &0 & 0 \\ 0 & 4&2i & 2 \\ 0 & -2i & 1 & -i \\ 0 & 2 & i & 1 \end{pmatrix}$,
$E_3 = \frac{1}{2}\begin{pmatrix} 0 &0 & 0 & 0 \\ 0 & 0 & 0 & 0 \\ 0 &0 & 1 & i \\ 0&0 & -i & 1 \end{pmatrix}$.

$E_1$ is formed using the first (non-zero) column of $E$ and notice that the number of initial zeros increase from  $E_1$ to $E_2$ to $E_3$.

 For a  mixed matrix  $A=\sum_ip_i  v_iv_i^*$ with $0<p \leq 1, \sum_ip_i=1$, as in Section \ref{mixed}. Now as shown $A=\sum_ip_iE_i$ where the $p_i$ are distinct,  each $E_i$ is a basic  idempotent and this expression for $A$ is unique apart from the order of the $E_i$. Each $E_i$ is a sum of of rank $1$ idempotents in a non-unique manner and such a sum may be obtained from Proposition \ref{idm} in theory and by Algorithm \ref{pro} in an efficient algorithm. 
 


Let  $u=(\al_1,\al_2, \ldots, \al_n)\T$ be a column vector.  Say $\st (u) =0$ if $\al_1\neq 0$ (if initial entry of $u$ is non-zero) and say  $\st (u) = r$ if $\al_1=\al_2=\ldots = \al_r$ but $\al_{r+1} \neq 0$ (if $\al_{r+1}$ is the first non-zero entry of $u$).  Let  $E=uu^*$ then define $\st (E)=\st (u)$. Thus $E$ has first $\st (E) $ columns and first $\st (E)$ rows  consisting entirely of zeros. The $\st$ measures the number of initial zeros of $u$ and the number of initial rows and columns of $E$. 

\begin{theorem}\label{nut} Let $E$ be an idempotent. Then $E$ is uniquely a sum $E=E_1+E_2+ \ldots + E_r$ where each $E_i$ is a rank $1$ idempotent and $\st (E_1) < \st (E_2) < \ldots < \st (E_r)$.
\end{theorem}
\begin{proof} Algorithm \ref{pro} shows that has such an expression for $E$ exists. 
The process involves  constructing first an idempotent $E_1$ such that 
 $E-E_1$ is an idempotent but also  that $\st (E-E_1) < \st (E)$.  Now  work with $(E- E_1)$ and proceed until the zero matrix is obtained. 

 Suppose now $E= E_1 + E_2 + \ldots + E_r = F_1+ F_2 + \ldots + F_s$ where 
$\st (E_1) < \st (E_2) < \ldots < \st (E_r)$ and $\st (F_1) < \st (F_2) < \ldots < \st (F_s)$  where the $E_i$ and $F_j$ are rank $1$ idempotents. 
 Then $r=s$ as this is the rank of $E$. Now look at the first non-zero column of $E_1$ which is the first non-zero column of $E$. From this get $E_1=F_1$ and then by induction that $E_i=F_i$ for $ i= 1,2,\ldots, r$.
 
 \end{proof} 

These results are applied below in Section \ref{mixed} to derive a uniqiue expression for a mixed matrix as the sum of rank $1$ matrices of the form given in Theorem \ref{nut}. 


\subsection{Mixed matrix}\label{mixed} 

Suppose 
$A=\sum_{i=1}^k p_i u_iu_i^*$  with  $p_i\neq 0$ where $\{u_1, u_2, \ldots, u_k\}$ are mutually orthogonal unit (column) vectors in $\C_n$.  
 Complete $\{u_1, u_2, \ldots, u_k\}$ to an orthogonal unit basis  $\{u_1, u_2, \ldots, u_n\}$  for $\C_n$. Let $E_i=v_iv_i^*$ and   $p_j= 0$ for $ j> k$. Now  $ A=\sum_{i=1}^n p_i u_iu_i^* = \sum_{i=1}^np_iE_i = \prod_{j=1}^k(I-E_j+p_jE_j)(I-E)$  
where $\{E_1, E_2, \ldots, E_k, E\}$ is a complete orthogonal set of idempotents. Basic matrices with the same  $p_i$ may be collected together. Note that $E=I-\sum_{j=1}^kE_j$ as $\{E_1, E_2,\ldots, E_k,E\}$ is a complete orthogonal set of idempotents. 

It is easy to check that $A^+=\prod_{j=1}^k(I-E_j+p_j^{-1}E_j)(I-E)$ is the {\em pseudo inverse} of $A$ which is useful for certain calculations. 

A particular case of this is where $A$ is  a {\em mixed density matrix} $A=\sum_{i=1}^k p_i u_iu_i^*$ where $\{u_1, u_2, \ldots , u_k\}$ are orthogonal unit vectors and $\sum p_i=1, 0<p_i\leq 1$. 
It is known that this expression for $A$ is not unique.

But $A= \sum_{i=1}^k p_i E_i$ where  the $E_i = u_iu_i^*$ are  rank $1$ idempotents. Then $A= \prod_{i=1}^k(I-E_i+p_iE_i)(I-F)$ where $F=I-\sum{j=1}^kE_j$.  Now by Lemma \ref{join} the basic products with the same $p_i$ can be gathered together to get $A= \prod_{j=1}^s(I-F_j+\hat{p}_jF_j)(I-F)$  where  $\{F_1, F_2, \ldots, F_s, F\}$ is a complete orthogonal set of idempotents with $F=I-\sum_{j=1}^sF_j$, the $\hat{p}_j$ are unique and $\sum_{j=1}^s\hat{p_j}=1, 0 < \hat{p}_j \leq 1$. This expression for $A$ is unique by Proposition  \ref{unique}. This means that $A= \sum_{i=1}^s\hat{p}_iF_i $ where the $\{F_j\}$ are orthogonal idempotents and the $\hat{p}_j$ is now a  unique expression for $A$.  Each $F_j$ is a sum $\sum_{i=1}^{s_j}u_{i_j}u_{i_j}^*$ of orthogonal rank $1$ idempotents and $F_j$ is of rank $s_j$. 

Note that $AF_j=\hat{p}_jF_j$ and $\hat{p}_j$ occurs with  multiplicity equal to rank $F_j$. 

Using Theorem \ref{nut} the following unique expression is obtained.

\begin{theorem} A density matrix $A$ is uniquely of the form
 $A= \sum_j \lam_j (E_{j,1}+ E_{j,2} + \ldots + E_{j,s_j})$ with distinct $\lam_j$ and  such $\st (E_{j,1}) < \st (E_{j,2}) < \ldots < \st (E_{j,s_j})$ for each $j$. 
\end{theorem}

\subsection{Roots}\label{root}

A nice formula for roots  of  normal matrices may now be obtained directly from their expression as a product of basic matrices. Roots of unitary matrices/quantum gates will then follow. 

Powers of normal  matrices are easy to obtain: 
  \begin{proposition}\label{roots} Let the normal  matrix  $B$ be  expressed as a product of basic matrices by 
  $B=\prod_{j=1}^k (I-E_j+\al_jE_j)$.
  Then $B^m = \prod_{j=1}^k (I-E_j+\al^mE_j)$.
  \end{proposition}
  
  \begin{proof} 
    The proof follows directly by induction and noting that the $E_j$ commute.
\end{proof}

Let $U= \prod_{j=1}^k(I-E_j+\al_jE_j)$ be the expression of
an $n^{th}$ root of $I_n$ as a product of basic matrices.  The $E_j$ are mutually orthogonal idempotents and the $\al_j\neq 0$.

  Then $U^n= \prod_{j=1}^k(I-E_j+\al^nE_j)=I_n$. Hence $\al^n=1$ and so $\al$ is an $n^{th}$ root of unity. Hence $\al_j = e^{i\theta_j}$ where $n\theta_j =2\pi r_j$ for $r_j\in \Z, 1\leq r_j\leq n-1$. Thus $\theta_j= \frac{2\pi r_j}{n}$ and so

  $$U= \prod_{j=1}^k(I-E_j+e^{i\frac{2\pi r_j}{n}}E_j)$$

  This gives normal matrices which  the $n^{th}$ roots of $I_n$. The $E_j$ can be any set of orthogonal idempotents and the $r_j$ can be any integers between $1$ and $n-1$.
  \begin{proposition} The $n^{th}$ roots of $I_n$ which are normal consist of matrices of the following type: $$U= \prod_{j=1}^k(I-E_j+e^{i\frac{2\pi r_j}{n}}E_j)$$
    where $\{E_1, E_2, \ldots, E_k\}$ is an orthogonal set of idempotents in $\C_{n\times n}$ and $r_j$ are integers satisfying $1\leq r_j\leq n-1$
  \end{proposition}

  
  Suppose $V= \{E_1,E_2, \ldots, E_k\}$ is a set of orthogonal symmetric
  idempotent matrices in $\C_n$. If $V$ is not complete then there exists $S= \{E_{k+1}, E_{k+2}, \ldots, E_s\}$   such that $\{E_1,E_2,\ldots, E_s\}$ is a complete orthogonal set of idempotents. The set $S$ is not unique but we shall refer to an $S$ as a  {\em complementary} set of $V$ and say that $S$ completes $V$. In particular $\{I-\sum_{j=1}^kE_j\}$ is the minimal set which completes $V$.

  Suppose  $U$ is the $n^{th}$ root of a basic matrix $W$ and $W= \prod_{j=1}^t(I-F_j+\be_jF_j)$ is the expression for $W$ as a product of basic matrices where the   $\be_j$ are all different.
  Let $U=\prod_{j=1}^s(I-E_j+\al_jE_j)$ be the expression for $U$ as a product of basic matrices. Some of  the $(I-E_j+\al_jE_j)$ may be $n^{th}$ roots of $I_n$ and we separate these out so that $U=\prod_{j=1}^k(I-E_j+\al_jE_j)P$ where $P$ is an $n^{th}$ root of $I_n$ and commutes with each $E_j$; reordering the $E_j$ may be necessary. Now $P$ is a product of basic matrices involving only idempotents in a complementary set of $\{E_1, E_2, \ldots, E_s\}$.

  Then  $U^n= \prod_{j=1}^k (I-E_j+\al^nE_j)P^n = \prod_{j=1}^k(I-E_j+\al^nE_j)$.
Now compare this with $W$. By uniqueness  and possible reordering get that $E_j = F_j$ and $\al_j^n=\be_j$. Thus if  $\al_j = r_je^{i\theta_j}$ and $\be_j=s_je^{i\gamma_j}$ then $r_j^n=s_j$ and $n\theta_j = \gamma_j+2\pi r_j$.  Now $n\_j=2\pi r_j +s_j$ for $0< s_j < 2\pi$ and so $e^{in\theta_j}=e^{is_j}$.



    However a normal  matrix has many $n^{th}$ roots and  we now proceed to directly find all these.

Suppose $U= \prod_{j=1}^k (I-E_j+\al_jE_j)$.  Now $\sum E_j$ may not be $I$ and in this case define $F= (I-\sum E_j)$. Then $\{E_1, E_2, \ldots, E_k, F\}$ is a complete orthogonal set of idempotents.  Now $I= I - F + e^0F$ and we need to attach this to $U$:
  $U= \prod_{j=1}^k (I-E_j+\al_jE_j)(I-F+e^0F)$. Suppose $n^{th}$ roots of $U$ are required. Apply the usual way of obtaining an $n^{th}$ root of unity in $\C$ and write:
  $U= \prod_{j=1}^k (I-E_j+\al_je^{i 2r_j\pi}E_j)(I-F+e^{2r\pi}F)$
  for positive integers $r_j,r$ from $0$ to $n-1$.
  Then $U^{\frac{1}{n}}= \{\prod_{j=1}^k (I-E_j+\al_j^{\frac{1}{n}}e^{i\frac{\theta_j+2r_j\pi}{n}}E_j)(I-F+e^{\frac{2r\pi}{n}}F)\}$
  and this is for any $r_j, r$ between $0$ and $n-1$. Here $\al^{\frac{1}{n}}$ refers to real $n^{th}$ of the the modulus of $\al$. This gives all the $n^{th}$ roots of $U$ and they are expressed in a unique form as a product of basic matrices. If the expression for $U$ has $k+1$ basic matrices, including $F\neq 0$, then get $n^{k+1}$ different $n^{th}$ roots of $U$.

  \subsubsection{Examples:} Consider the examples \ref{example}. 
\begin{enumerate} \item Example \ref{oneone} uses the matrix 
$U= \begin{ssmatrix} 0 & 1 \\ 1 & 0 \end{ssmatrix}$ with  $E_1 = \frac{1}{2}\begin{ssmatrix} 1 & 1 \\ 1 & 1 \end{ssmatrix},
E_2 = \frac{1}{2}\begin{ssmatrix} 1 & -1 \\ -1 & 1 \end{ssmatrix}$ and   $U = I - E_2 +e^{\pi}E_2$.
The square roots of $U$  are then 
$\{I - E_2 +e^{i\frac{\pi}{2}}E_2, I - E_2 +e^{i\frac{3\pi}{2}}E_2\}$ plus a product of these using  any square roots of $I$ obtained from a complementary set of $\{E_2\}$. The only complementary set of $\{E_2\}$ is $\{E_1\}$. A basic matrix  involving $E_1$, besides the identity, which is a square root of $I$ is $I-E_1+e^{i\pi}E_1$. Thus the square roots of $U$ are
$\{(I - E_2 +e^{i\frac{\pi}{2}}E_2)(I-E_1+ e^{ik\pi}E_1), (I - E_2 +e^{i\frac{3\pi}{2}}E_2)(I-E_2+e^{is\pi}E_2)\}$ where $k,s$ can be $0$ or $1$. 
  
\item Example \ref{two} uses  the orthogonal matrix 
$U=\begin{ssmatrix} \cos \theta & \sin \theta \\ -\sin \theta & \cos
    \theta \end{ssmatrix}$. Then  
$U=(I-E_1+e^{i\theta}E_1)(I-E_2+e^{-i\theta}E_2)$ 
with $\{E_1 = \frac{1}{2}\begin{ssmatrix}1 & -i \\ i & 1
\end{ssmatrix}, E_2 = \frac{1}{2}\begin{ssmatrix}1 & i \\ -i & 1
\end{ssmatrix} \}$.
Here $\{E_1,E_2\}$ is a complete orthogonal set of idempotents so there is no complementary set to consider. 

The third roots of $U$ are: 
$\{ (I-E_1+e^{i\frac{\theta+2r\pi}{3}}E_1)(I-E_2+e^{-i\frac{\theta+2s\pi}{3}}E_2)\}$ where $r,s$ can have values from $\{0,1,2\}$.   

\end{enumerate}

\section{Logic gates}\label{open}
Expressions as products of basic matrices are found 
for  common quantum logic gates. 

\begin{itemize}
 \item{\bf Hadamard Gate} The unitary matrix of the the Hadamard Gate is $H=\frac{1}{\sqrt{2}}\begin{pmatrix} 1 & 1 \\ 1 & -1 \end{pmatrix} $.
Then  $H^2=I$ and $H$ has eigenvalues $\{1, -1\}$.
 Thus by Proposition \ref{cor}  $H$ can be expressed as $H= (I- E-E)$. 
 Solving for $E$ get
  $E = \frac{1}{2}\begin{pmatrix} 1 - \frac{1}{\sqrt{2}} &
	       -\frac{1}{\sqrt{2}} \\ -\frac{1}{\sqrt{2}} & 1 +
 \frac{1}{\sqrt{2}}\end{pmatrix} $.
Thus    $H= (I- E+e^{i\pi}E) = (I-E-E)= (I-2E)$
    expresses $H$ as a product of (minimal) basic unitary matrices where $E$ is as above. $H$ is itself a minimal basic matrix.
    Then $\{E,F\}$ with $F=I-E$ is  a complete orthogonal set of idempotents. Now $HE=-E$ and $HF=F$ so $H$ `reverses' $E$ and leaves $F$ alone.  
    
    Roots of $H$ are immediate: For example a square root of $H$ is
    $(I- E+ e^{i\frac{\pi}{2}}E) = (I- E -iE)$. All the square roots of $H$ may be obtained  by methods of  Section \ref{root} and these are $(I-E+e^{i\frac{r\pi}{2}}E)(I- F+ e^{is\pi}F)$ where $r\in\{1,3\}, s\in\{0,1\}$.
The action of $H$ on the idempotents is given by $HF=F, HE=e^{i\pi}E$;  $F$ is left
    alone by $H$ and $E$ is changed  to $e^{i\pi}E$. 
 



     \item {\bf Pauli gates:}
The Pauli gates are symmetric unitary matrices and hence satisfy $X^2=I$ and have eigenvalues $\{1,-1\}$.  Thus by Proposition \ref{cor} they have the form $(I-E+e^{i\pi}E)= I-2E$ for an idempotent $E$. 
It is now a matter of establishing $E$ in these cases. 
    
\begin{itemize} \item{\bf Pauli X-gate}  is $X= \begin{pmatrix} 0 & 1 \\ 1 & 0 \end{pmatrix}$ and has eigenvalues $\{1,-1\}$. 
        Thus $I-2E= X$ for a symmetric idempotent $E$. Solving for $E$ gives  giving that $ E=  \begin{pmatrix}\frac{1}{2} & - \frac{1}{2} \\ -\frac{1}{2} & \frac{1}{2}\end{pmatrix}$.
Hence  $X 
= (I-E+e^{i\pi}E)$.
In this form a square root of $X$ is easy to find and one such is  $(I-E+e^{i\frac{\pi}{2}}E)$. By section \ref{root} all square roots of $X$ are  $\{\sqrt{X}\} = \{(I-E+e^{i\frac{r\pi}{2}}E)(I-F+e^{i\frac{s\pi}{2}}F)$ where $r\in\{1,3\}, s\in\{0,1\}$ and $F=I-E =
\begin{pmatrix} \frac{1}{2} &\frac{1}{2} \\\frac{1}{2} & \frac{1}{2}\end{pmatrix}$.

\item{\bf Pauli Y-gate} is $Y=\begin{pmatrix} 0 & -i \\ i & 0 \end{pmatrix}$.
Then $Y = I-E+e^{i\pi}E$ where $E= \begin{pmatrix}
\frac{1}{2} & \frac{i}{2} \\ -\frac{i}{2} & \frac{1}{2}
\end{pmatrix}$.

Roots of Y-gate are easy to obtain:  for example
$(I-E+e^{i\frac{\pi}{4}}E)$ is a fourth root and  all fourth roots may be directly found by structures  of Section \ref{root}. 

\item{\bf Pauli Z-gate} is $Z=\begin{pmatrix} 1&0\\ 0 & -1\end{pmatrix}$. Then
$Z= 
    (I-E+e^{i\pi}E)= I-2E$ where $E= \begin{pmatrix} 0&0\\ 0 & 1\end{pmatrix}$.

\end{itemize}
The Pauli gates are themselves (minimal) basic matrices. 
        
   \item{\bf Phase shift gate:}     The Phase shift unitary matrix is $Ph=\begin{pmatrix} 1&0\\ 0 & e^{i\theta}\end{pmatrix}$ and has eigenvalues $\{1,e^{i\theta}\}$. Applying the general result gives 
      $Ph=(I-F+e^{i\he}F)$ where $F=\begin{pmatrix} 0&0\\ 0 & 1\end{pmatrix}$. This expresses $Ph$ as a minimal basic unitary matrix.  Roots of $Ph$ are easily obtained by the method of section \ref{root}. For example a  third root of $Ph$ is $(I-F+e^{i\frac{\he}{3}}F)$ but there are others, see Section \ref{root}. 

\item{\bf NOT and Square root of NOT:}
The NOT gate is $NOT= \begin{pmatrix} 0 & 1 \\ 1 & 0 \end{pmatrix}$.
             This has eigenvalues $\{1,-1\}$.  Then  $NOT = (I - E + e^{i\pi}E)$.where $E=
			 \frac{1}{2}\begin{pmatrix}1 & -1 \\ -1 & 1
				    \end{pmatrix}$.  
              
              Then  (one value) of $ \sqrt{NOT} = I - E + e^{i\frac{\pi}{2}}E = \frac{1}{2}\begin{pmatrix} 1+i & 1-i \\ 1 - i & 1+i \end{pmatrix} $.
              For example it is possible to define $\sqrt{\sqrt{NOT}} = I- E+
              e^{i\frac{\pi}{4}}E$. 
                
            \item{\bf Swap, Square root of Swop:}
The Swap gate has matrix
$ Sw= \begin{ssmatrix} 1 &0&0&0 \\0&0&1&0\\0&1&0&0 \\ 0&0&0&1 \end{ssmatrix}$.
This has eigenvalues $\lam = 1$, (three times) and $\lam = -1$.
Then with  $E= \begin{ssmatrix} 0 &0&0&0 \\0&\frac{1}{2}&-\frac{1}{2}&0\\0&-\frac{1}{2}&\frac{1}{2}&0 \\ 0&0&0&0 \end{ssmatrix}$ it is seen that
$Sw= I- E +e^{i\pi}E$.

Then a square root is $\sqrt{Sw} = I- E + iE = I - E + e^{i\frac{\pi}{2}}E$, where $E$ is the idempotent obtained for $Sw$ and this is  $\begin{ssmatrix} 1 & 0&0&0 \\ 0 & \frac{1}{2} (1 + i) & \frac{1}{2} (1 - i) & 0 \\ 0 & \frac{1}{2} (1 - i) & \frac{1}{2} (1 +i) &0 \\ 0&0&0&1\end{ssmatrix} $ as usually given.

$(I-E+e^{i\frac{\pi}{4}}E) $ could be considered as `a square root of a square root of Swap'.



  \item{ \bf CNOT matrix:} CNOT matrix   is $CNOT= \begin{ssmatrix} 1&0&0&0 \\ 0 &1&0 &0 \\ 0&0&0&1 \\ 0&0&1&0\end{ssmatrix} $. This has eigenvalues $1$, three times, and eigenvalue $-1$. Let $E= \begin{ssmatrix} 0&0&0&0 \\ 0&0&0&0 \\ 0&0& \frac{1}{2} & -\frac{1}{2} \\ 0&0&-\frac{1}{2}& \frac{1}{2}\end{ssmatrix}$. Then $CNOT = (I-E+e^{i\pi}E)$. $E$ has rank $1$ and $F=I-E$, corresponding to eigenvalue $1$, has rank $3$. All the square roots of CNOT  are  $\{(I-E+e^{i\frac{\pi}{2}}E), (I-E+e^{i\frac{3\pi}{2}}E), (I-E+e^{i\frac{\pi}{2}}E)(I-F+e^{i\pi}F), (I-E+e^{i\frac{3\pi}{2}}E)(I-F+e^{i\pi}F)\}$.
  \item{\bf Bell non-symmetric:} The Bell non-symmetric matrix is $B=\begin{pmatrix} 0&1 \\ -1 & 0\end{pmatrix}$. This has eigenvalues $\{i,-i\}$. Then it is easy to show that $B= (I-E-iE)(I-F+iF)$ where $E= \frac{1}{2}\begin{pmatrix}1 &i\\-i & 1\end{pmatrix}, F= \frac{1}{2}\begin{pmatrix}1&-i \\ i&1 \end{pmatrix} = I-E$.
    Thus $B=(I-E-iE)(I-F+iF)= (I-E+e^{i\frac{3\pi}{2}}E)(I-F+e^{i\frac{\pi}{2}}F)$. Note that $\{E,F\}$ is complete. 
\end{itemize}
\section{Build matrices}\label{build1} 
    Normal  matrices are built from basic matrices of the form $(I-E+\al E)$; unitary matrices or quantum logic gates are  built from basic matrices of the form $(I-E+e^{i\theta}E)$. 
A particular property may  be required; for example a quantum logic gate which is  an $n^{th}$ root and/or a quantum logic gate with particular eigenvalues and  multiplicities may be required. 
Requirements may be met in a constructive manner.

It is  shown  in Proposition \ref{cor} that a symmetric (unitary matrix)/(quantum logic gate) is  of the form $(I-2E)$ for a symmetric  idempotent $E$. This makes building symmetric unitary matrices/logic gates particularly easy.    


    For example the matrix $U=\frac{1}{2}\begin{ssmatrix} 1&1&1&-1 \\ 1&1&-1&1 \\1&-1 &1 &1 \\ -1 &1&1&1 \end{ssmatrix}$ is the unitary matrix related to the  Hadamard matrix $2U$.  $U$ is symmetric and so  $U=(I-2E)$ where $E$ is easily shown to be $E=\frac{1}{4}\begin{ssmatrix} 1&-1&-1& 1 \\ -1 & 1&1 &-1\\ -1&1&1&-1 \\ 1&-1&-1 &1 \end{ssmatrix}$. It is clear that $\rank E=2$ and thus $\rank (I-E)= 2$ also. A square root of $U$ is $(I-E+e^{i\frac{\pi}{2}}E)$ but there are three other square roots which are easily worked out using the methods developed in Section \ref{root}. $U$ is not separable and roots of $U$ are not separable. 

    A unitary matrix $U$ which is also a root of $I$ may be built by the methods of Section \ref{roots}; in this case $U^n=I$ and $U^*= U^{n-1}$.
Suppose for example quantum logic gates (unitary matrices)  in $C_{4\times 4}$ are  required which are  $4^{th}$ roots of $I_4=I$. Clearly  $U= (I - E + e^{i\frac{2r\pi}{4}}E)$ where $E$ is any idempotent and $r\in \{1,2,3\}$, is a basic unitary quantum gate which is a $4^{th}$ root of $I$. Now $v_1=\frac{1}{2}(1,1,1,-1)\T, v_2=\frac{1}{2}(1,1,-1,1)\T, v_3=\frac{1}{2}(1,-1,1,1)\T, v_1=\frac{1}{2}(-1,1,1,1)\T$ is an orthonormal basis for $\C_4$ and these give an orthogonal complete set of idempotents  $E_1=v_1v_1^*, E_2= v_2v_2^*, E_3= v_3v_3^*, E_4=v_4v_4^*$. From these,  gates which are $4^{th}$ roots of $I$ may be formed:
  $\{ (I - E_1 + e^{i\frac{2r_1\pi}{4}}E_1)(I - E_2 + e^{i\frac{2r_2\pi}{4}}E_2)(I - E_3 + e^{i\frac{2r_3\pi}{4}}E_3)(I - E_4 + e^{i\frac{2r_1\pi}{4}}E_4)\}$ 
where $r_j\in \{0,1,2,3\}$. (If all $r_j=0$ then the identity matrix is obtained.)



Normal including unitary and Hermitian matrices may naturally be built from a set of unit orthogonal vectors.  Let $\{u_1,u_2, \ldots, u_k\}$ be a set of orthogonal unit column vectors in $\C_n$   
     and set  $E_i=u_iu_i^*$. Build $B=  (I-E_1+\al_1E_1)(I-E_2+\al_2E_2)\ldots (I-E_k+\al_kE_k)$ which is then normal. The basic matrices with the same scalars may be amalgamated by Lemma \ref{join} to get a unique expression $B=
    (I-F_1+\be_1F_1)(I-F_2+\be_2F_2)\ldots (I-F_s+\be_sF_s)$ where the $\be_i$ are distinct. Each $F_j$ is a sum $\sum u_{i_j}u_{i_j}^*$ of a subset of $\{u_1u_1^*, u_2u_2^*, \ldots, u_ku_k^*\}$. 
    Note that $S=\{u_1,u_2,\ldots, u_k\}$ need not be a full basis for $\C_n$; in this case $\{E_1,E_2,\ldots, E_k\}$ is not a complete set of orthogonal idempotents. By letting $E= I- \sum_{i=1}^{k}E_i$ a complete set $\{S \cup E\}$ of orthogonal idempotents is obtained. To build a unitary matrix take $\al_j=e^{i\theta_j}$ and to build a Hermitian matrix take the $\al_j$ to be real. If no $\al_i=0$ then no $\be_j=0$ and $B$ is invertible with inverse $B^{-1} = (I-F_1+\be_1^{-1}F_1)(I-F_2+\be_2^{-1}F_2)\ldots (I-F_{s}+\be_{s}^{-1}F_{s})$. If some $\be_j=0$, say $b_s=0$, then $B=(I-F_1+\be_1F_1)(I-F_2+\be_2F_2)\ldots (I-F_{s-1}+\be_{s-1}F_{s-1})(I-F_s)$. In this case the pseudo inverse of $B$ is $(I-F_1+\be_1^{-1}F_1)(I-F_2+\be_2^{-1}F_2)\ldots (I-F_{s-1}+\be_{s-1}^{-1}F_{s-1})(I-F_s)$. 
    

    Sets of orthogonal unit column $n\ti 1$ vectors may be  obtained from $n^{th}$ roots of unity 
    from which unitary matrices and logic gates for  various requirements may be  built; details are omitted.

    Comparisons may be made with the famous 1D unique building blocks  for paraunitary matrices  due to Belvitch and Vaidyanathan, see \cite{1D} pages 302-322.








\begin{thebibliography}{99}

    \bibitem{idemrank} Oskar M. Baksalary, Dennis S. Bernstein,
	Götz Trenkler, ``On the equality between rank and trace of an
	idempotent matrix'', Applied Mathematics and Computation, 217,
	4076-4080, 2010. 

\bibitem{tedbarry} Ted Hurley and Barry Hurley, ``Paraunitary matrices and group rings'', Intl. J.  Group Theory, Vol. 3, no.1, pp 31-56, 2015.  


 \bibitem{nielsen} Michael A. Nielsen,
Isaac Chuang, {\em Quantum Computation and Quantum Information}, Cambridge University Press, Cambridge UK, 2010.  

\bibitem{strang} Gilbert Strang and Truong Nguyen, {\em Wavelets and Filter Banks}, Wesley-Cambridge Press, 1997.

\bibitem{1D}  P. P. Vaidyanathan, {\em Multirate Systems and Filterbanks}, Prentice-Hall, 1993.

\end{thebibliography}
\end{document}